\documentclass[12pt]{amsart}
\usepackage{latexsym}
\usepackage{amsmath}
\usepackage{amssymb}
\usepackage{amsthm}
\usepackage[colorlinks=true,linkcolor=blue,citecolor=blue,urlcolor=green]{hyperref}
\usepackage[margin=2.8cm]{geometry}
\usepackage{fourier}

\newtheorem{theorem}{Theorem}[section]

\newtheorem{proposition}[theorem]{Proposition}

\theoremstyle{definition}

 \newtheorem{remark}[theorem]{Remark}

\theoremstyle{remark}

\DeclareMathOperator{\id}{id}

\begin{document}
\title{Summation of coefficients of polynomials on $\ell_{p}$ spaces}
\author[V. Dimant]{Ver\'{o}nica Dimant}
\author[P. Sevilla-Peris]{Pablo Sevilla-Peris}

\thanks{The first author was partially supported by CONICET PIP 0624 and ANPCyT PICT 1456. The second author was supported by MICINN Project MTM2011-22417 and partially by grant GVA-BEST/2013/113 and project UPV-SP201207000.}

\subjclass[2010]{}

\address{Departamento de Matem\'{a}tica, Universidad de San
Andr\'{e}s, Vito Dumas 284, (B1644BID) Victoria, Buenos Aires,
Argentina and CONICET.} \email{vero@udesa.edu.ar}

\address{Instituto Universitario de Matem\'{a}tica Pura y Aplicada and DMA, ETSIAMN, Universitat Polit\`{e}cnica de Val\`{e}ncia, cmno Vera s/n, 46022, Valencia, Spain} \email{psevilla@mat.upv.es}

\begin{abstract}
 We investigate the summability of the coefficients of $m$-homogeneous polynomials and $m$-linear mappings defined on $\ell_{p}$-spaces. In our research we obtain results
on the summability of the coefficients of $m$-linear mappings defined on $\ell_{p_{1}} \times \cdots \times \ell_{p_{m}}$. The first results in this respect go back to Littlewood and
Bohnenblust and Hille (for bilinear and $m$-linear forms on $c_{0}$) and Hardy and Littlewood and Praciano-Pereira (for bilinear and $m$-linear forms on arbitrary $\ell_{p}$-spaces).
Our results recover and in some case complete these old results through a general approach on vector valued $m$-linear mappings.
\end{abstract}

\maketitle

\section{Introduction}

Every $m$-homogeneous polynomial $P$ defined on $\ell_{p}$ with values on some Banach space $X$ defines a family of coefficients $\big(c_{\alpha}(P) \big)_{\alpha \in \Lambda_{m}}$ (here $\Lambda_{m}$ denotes the set of multi-indices that eventually become $0$
such that $\vert \alpha \vert = \sum_{j} \alpha_{j} =m$) in the following way: consider $T$ the unique symmetric $m$-linear form associated to $P$ then, for $\alpha = (\alpha_{1} , \ldots , \alpha_{n}, 0 \ldots)$
with $\alpha_{1} + \cdots + \alpha_{n}=m$ we have
\[
 c_{\alpha}(P) = \frac{m!}{\alpha_{1} !  \cdots  \alpha_{n}!} T(e_{1}, \stackrel{\alpha_{1}}{\ldots} , e_{1} , \ldots , e_{n} ,  \stackrel{\alpha_{n}}{\ldots} , e_{n}) \,.
\]
Our interest is to investigate the summability properties of these coefficients. As consequences of results due to Aron and Globevnik for polynomials on $c_{0}$ \cite[Corollary~1.4]{ArGl89} and of Zalduendo for general $\ell_{p}$ spaces
\cite[Corollary~1]{Za93} we have that there exists a constant $C>0$ such that for every $m$-homogeneous polynomial $P: \ell_{p} \to \mathbb{C}$ (with $m<p < \infty$) we have
\[
 \Big( \sum_{i=1}^{\infty} \vert P(e_{i}) \vert^{\frac{p}{p-m}} \Big)^{\frac{p-m}{p}} \leq C \Vert P \Vert \, ,
\]
and the exponent is optimal (if the polynomial is defined on $c_{0}$ then the exponent is $1$). This can be seen as summing the coefficients over the family of indices $\alpha = (0 , \ldots , 0,m, 0 , \ldots )$. If we sum over all coefficients the situation is
pretty well understood for polynomials on $c_{0}$ (or $\ell_{\infty}$) by the results by Bohnenblust and Hille \cite{BoHi31} for scalar-valued polynomials and by Defant and Sevilla-Peris \cite{DeSe09} in the vector-valued setting. Following the spirit of  \cite{DeSe09} we focus on the coefficients of
polynomials defined on some $\ell_{p}$ space with values on some other $\ell_{u}$, computing the norm of the coefficients on a bigger $\ell_{q}$. Then the main result of the paper is the following.
\begin{theorem} \label{main polin}
 Let $1 \leq p \leq \infty$ and $1\leq u \leq q \leq \infty$. Then there is $C>0$ such that, for every continuous $m$-homogeneous polynomial $P: \ell_{p} \to \ell_{u}$ with coefficients $\big(c_{\alpha}(P) \big)$ we have
\[
  \Big( \sum_{\alpha \in \Lambda_{m}} \Vert c_{\alpha} (P) \Vert_{\ell_{q}}^{\rho} \Big)^{1/\rho} \leq C \Vert P \Vert \,.
\]
where $\rho$ is given by
\begin{enumerate}
  \item\label{polin-i} If $1 \leq u \leq q \leq 2$, and
      \begin{enumerate}
        \item\label{polin-i-a} if $\frac{mqu}{q-u} < p \leq \infty$, then $\rho = \frac{2m}{m+2(1/u - 1/q - m/p)}$.
	\item\label{polin-i-b} if $ \frac{2muq}{uq+2q-2u} < p \leq \frac{mqu}{q-u}$, then
						  $\rho = \frac{2}{1+2(1/u - 1/q - m/p)}$.
      \end{enumerate}
  \item\label{polin-ii} If  $1 \leq u \leq 2 \leq q$, and
      \begin{enumerate}
        \item\label{polin-ii-a} if $\frac{2mu}{2-u} < p \leq \infty$, then $\rho = \frac{2m}{m+2(1/u - 1/2 - m/p)}$.
	\item\label{polin-ii-b} if $ mu < p \leq  \frac{2mu}{2-u}$, then
						  $\rho = \frac{1}{1/u -  m/p}$.
      \end{enumerate}
  \item\label{polin-iii} If  $2 \leq u \leq q \leq \infty$ and $ mu < p \leq \infty$, then $\rho = \frac{1}{1/u -  m/p}$.
\end{enumerate}
Moreover, the exponents in the cases \eqref{polin-i-a}, \eqref{polin-ii-b} and \eqref{polin-iii} are optimal. Also, the exponent in \eqref{polin-i-b} is optimal for  $p >2m$.
\end{theorem}

\noindent We will approach the problem through multilinear mappings. Given an $m$-homogeneous polynomial $P$ we take
$T$ the associated symmetric $m$-linear and denote $a_{i_{1} \ldots i_{m}} = T(e_{i_{1}} , \ldots , e_{i_{m}})$. Since
$\Vert T \Vert \leq e^{m} \Vert P \Vert$ (see e.g. \cite[Corollary~1.8]{Di99}), each time that an inequality of the type
\begin{equation} \label{eq multi}
 \Big( \sum_{i_{1} \ldots i_{m}} \Vert a_{i_{1} \ldots i_{m}} \Vert^{t} \Big)^{1/t} \leq C \Vert T \Vert
\end{equation}
holds for every $m$-linear mapping
we automatically have  an equivalent inequality (with the same exponent) for all $m$-homogeneous polynomials (see \cite[Lemma~5]{DeSe09} for more details). Littlewood showed in \cite{Li30} that an inequality like \eqref{eq multi} holds
with $t=4/3$ for bilinear forms on $c_{0}$. This result was generalised by Bohnenblust and Hille \cite{BoHi31} to $m$-linear forms on $c_{0}$ and by Hardy and Littlewood \cite{HaLi34} to bilinear forms on $\ell_{p} \times \ell_{q}$. In all these
results the exponents in the respective inequalities were shown to be optimal. Praciano-Pereira gave in \cite{PP81}  inequalities for multilinear forms defined on $\ell_{p_{1}} \times \cdots \times \ell_{p_{m}}$, but he did not cover all possible cases
and he did not deal with the optimality of the exponents. Recently there have been also some results on vector valued multilinear mappings defined on $c_{0}$ \cite{DeSe09,DePoSc10}.\\
Our result for polynomials will follow from the following more general result on $m$-linear mappings, that is our second main result.
\begin{theorem} \label{main main}
Let $Y$ be a cotype $q$ space and $v:X \to Y$ an $(r,1)$--summing operator (with $1 \leq r \leq q$). For $1 \leq p_{1}, \dots,p_{m} \leq \infty$ with $\frac{1}{p_{1}} + \cdots + \frac{1}{p_{m}} < \frac{1}{r}$
we define
\[
\frac{1}{\lambda} =  \frac{1}{r} - \Big(\frac{1}{p_{1}} + \cdots + \frac{1}{p_{m}} \Big) \qquad \text{ and } \qquad
\frac{1}{\mu} =  \frac{1}{m \lambda} + \frac{m-1}{mq} \,.
\]
Then there exists $C>0$ such that, for every $m$-linear $T : \ell_{p_{1}} \times \cdots \times \ell_{p_{m}} \to X$
with coefficients $(a_{i_{1}, \ldots , i_{m}})$ we have
\begin{enumerate}
\item\label{main-i} If $\lambda \geq q$, then $ \displaystyle
\Big( \sum_{i_{1}, \ldots, i_{m}=1}^{\infty} \Vert v a_{i_{1}, \ldots , i_{m}} \Vert^{\lambda}\Big)^{1/\lambda} \leq C \Vert T \Vert$.

\item\label{main-ii} If $\lambda < q$, then $ \displaystyle
\Big( \sum_{i_{1}, \ldots, i_{m}=1}^{\infty} \Vert v a_{i_{1}, \ldots , i_{m}} \Vert^{\mu}\Big)^{1/\mu} \leq C \Vert T \Vert$.
\end{enumerate}
\end{theorem}

\noindent We can rewrite
\[
 \frac{1}{\mu} = \frac{q+(m+1)r}{mrq} + \frac{1}{m} \Big(\frac{1}{p_{1}} + \cdots + \frac{1}{p_{m}}  \Big ) \, ,
\]
then we easily see that doing $p_{1}= \ldots = p_{m} = \infty$ we recover (with the same exponent) \cite[Corollary~5.2]{DePoSc10}. On the other hand,
taking $X=Y=\mathbb{C}$ and $v$ the identity we recover the classical result of Hardy and Littlewood \cite{HaLi34} in the bilinear case and
we recover and complete with the remaining cases the results in \cite{PP81} (see Proposition \ref{praciano} below).

\section{Definitions and preliminaries}
We collect now some of the main definitions and results that we will be using along the paper.  All the coming spaces will be complex Banach spaces.
The open unit ball of $X$ will be denoted by $B_{X}$ and the dual of $X$ by $X^{*}$.\\
\noindent The space of continuous $m$-linear mappings on $X_{1} \times \cdots \times X_{m}$ with values in $Y$ will be denoted by
$\mathcal{L} (^{m} X_{1}, \ldots , X_{m} ; Y)$. With the norm
\[
\Vert T \Vert = \sup \{ \Vert T(x_{1}, \ldots , x_{m}) \Vert \colon x_{j} \in B_{X_{j}} , \, j=1 , \ldots , m\}
\]
it is a Banach space.\\
Every $m$-linear mapping $T$ defined on $\ell_{p_{1}} \times \cdots \times \ell_{p_{m}}$ defines a set of coefficients given by
$a_{i_{1} \ldots i_{m}} = T(e_{i_{1}} , \ldots , e_{i_{m}})$.\\
A mapping $P:X \to Y$ is a (continuous) $m$--homogeneous polynomial if there exists a (continuous) $m$-linear mapping $T: X\times \cdots \times X \to Y$
such that $P(x)=T(x, \dots, x)$ for every $x$. The space of continuous $m$-homogeneous polynomials is denoted by $\mathcal{P}(^{m} X;Y)$
and with the norm $\Vert P \Vert = \sup\{ \Vert P(x) \Vert : x \in B_{X} \}$ is a Banach space. Each polynomial has a unique associated symmetric $m$-linear mapping.\\
Given $1\leq p \leq \infty$, the conjugate $p'$ is defined by $1=\frac{1}{p} + \frac{1}{p'}$.\\

\noindent A Banach space has cotype $q$
(see e.g. \cite[Chapter~11]{DiJaTo95}) if there exists a constant $C>0$ such that for every finite choice of elements $x_{1}, \dots , x_{N} \in X$
\[
\Big( \sum_{k=1}^{N} \Vert x_{k} \Vert^{q} \Big)^{1/q}
\leq C \bigg( \int_{0}^{1} \Big\Vert \sum_{k=1}^{N} r_{k} (t) x_{k} \Big\Vert^{2} dt \bigg)^{1/2} \,,
\]
where $r_{k}$ is the $k$--th Rademacher function. The smallest constant in this inequality is denoted by $C_{q}(X)$. Recall that $\ell_q$ has cotype $\max\{q,2\}$.\\
We will use repeatedly the following easy fact: whenever $X$ has cotype $q$ and $s\ge q$ then $\ell_s^n(X)$ has cotype $s$ with $C_s(\ell_s^n(X))\le C_s(X)$.

\noindent An operator between Banach spaces $v:X \to Y$ is $(r,s)$--summing (with $s\leq r \leq \infty$) \cite[Chapter~10]{DiJaTo95} if there exists $C>0$ such that for
every finite choice $x_{1}, \dots , x_{N} \in X$
\[
\Big( \sum_{k=1}^{N} \Vert v x_{k} \Vert^{r} \Big)^{1/r}
\leq C \sup_{x^{*} \in B_{X^{*}}} \Big(  \sum_{k=1}^{N} \vert x^{*} (x_{k}) \vert^{s} \Big)^{1/s} \,,
\]
The smallest constant in this inequality is denoted by $\pi_{r,s}(v)$.\\
A straightforward computation shows that an operator $v:X \to Y$ is $(r,s)$--summing if and only if there exists $C>0$ such that  for every $n$ and every operator $T:\ell_{s'}^{n} \to X$ we have
\begin{equation} \label{eq r1}
 \Big( \sum_{k=1}^{n} \Vert vT(e_{k}) \Vert_{Y}^{r} \Big)^{1/r} \leq C \Vert T \Vert \,.
\end{equation}
Also, it is well known that if a Banach space $X$ has cotype $q$, then the identity $\id:X \to X$ is $(q,1)$--summing (see e.g. \cite[Theorem~11.17]{DiJaTo95}).\\
We will be using some facts about $(r,s)$--summing operators. The first one is the \textbf{Inclusion Theorem} \cite[Theorem~10.4]{DiJaTo95}: if $s_{1} \leq s_{2}$,
$r_{1} \leq r_{2}$ and $\frac{1}{s_{1}}- \frac{1}{r_{1}} \leq \frac{1}{s_{2}}- \frac{1}{r_{2}}$ then every $(r_{1},s_{1})$--summing operator
is $(r_{2},s_{2})$--summing and $\pi_{r_{2},s_{2}} (v) \leq \pi_{r_{1},s_{1}} (v) $.\\
Our second main fact are the celebrated \textbf{Bennett--Carl inequalities} \cite{Be73,Ca74}, that describe precisely how summing the inclusion mappings
between $\ell_{p}$--spaces are:   given  $1 \leq u \leq q \leq \infty$ define the number
\[
r =
\left\{
\begin{array}{ll}
\displaystyle \frac{2}{1 + 2 ( \frac{1}{u}- \frac{1}{q}  )} & \mbox{ if }  q < 2 \\
u & \mbox{ if } q \geq 2\,.
\end{array}
\right.
\]
Then the inclusion $\id: \ell_{u} \hookrightarrow \ell_{q}$ is $(r,1)$--summing and this $r$ is optimal.\\

\noindent We will use the normed theory of tensor products as presented in \cite{DeFl93}. The injective tensor norm will be denoted by $\varepsilon$.
An operator $v:X \to Y$ is $(r,s)$--summing if and only if there is $C>0$ such that $\Vert \id \otimes v : \ell_{s}^{n} \otimes_{\varepsilon} X
\to \ell_{r}^{n} (Y) \Vert \leq C$ for every $n \in \mathbb{N}$; in this case $\pi_{r,s}(v) = \sup_{n} \Vert \id \otimes v : \ell_{s}^{n} \otimes_{\varepsilon} X
\to \ell_{r}^{n} (Y) \Vert$.\\

\noindent Finally, we will be dealing with sums over indices $(i_{1}, \ldots , i_{m}) \in \{1, \ldots , n\}^{m}$. The symbol $\sum_{[i_{k}]}$ will mean that we are fixing
the $k$--th index and summing over all the rest.\\
The cardinal of a set $A$ will be denoted by $\sharp A$.

\section{Proof of Theorem~\ref{main main}}
The main tool for the proof of the main result will be the following inequality for mixed sums. For scalar valued mappings this is  \cite[(1.2.8)]{HaLi34} in the
bilinear case and \cite[Theorem~A]{PP81} in the $m$-linear case. Our proof follows the guidelines of \cite{PP81} and we present here an adapted version.

\begin{proposition} \label{main tool}
Let $Y$ be a cotype $q$ Banach space and $v:X \to Y$ an $(r,1)$--summing operator (with $1 \leq r \leq q$). Assume $1 \leq p_{1}, \ldots , p_{m} \leq \infty$ are such that $\frac{1}{p_{1}} + \cdots + \frac{1}{p_{m}}
< \frac{1}{r} - \frac{1}{q}$ and let $\frac{1}{\lambda} = \frac{1}{r} - \big( \frac{1}{p_{1}} + \cdots + \frac{1}{p_{m}} \big)$. Then for every continuous
$m$-linear mapping $T: \ell_{p_{1}} \times \cdots \times \ell_{p_{m}} \to X$ we have, for each $j=1 , \ldots, m$,
\[
\bigg( \sum_{i_{j}} \Big( \sum_{[i_{j}]} \Vert v T(e_{i_{1}}, \ldots, e_{i_{m}}) \Vert^{q} \Big)^{\lambda/q} \bigg)^{1/\lambda} \leq \left(\sqrt{2}C_q(Y)\right)^{m-1} \pi_{r,1}(v) \Vert T \Vert \, .
\]
\end{proposition}
\begin{proof}
If $p_1=\cdots =p_m=\infty$, then $\lambda=r$ and proceeding as in \cite[Lemma~2]{DeSe09} we easily get
\[
 \left(\sum_{i_j}\left(\sum_{[i_j]}\|vT(e_{i_1},\dots,e_{i_m})\|^q\right)^{r/q}\right)^{1/r} \le \left(\sqrt{2}C_q(Y)\right)^{m-1} \pi_{r,1}(v) \Vert T \Vert
\]
for every $m$-linear $T:\ell_{\infty} \times \cdots \ell_{\infty} \to X$ and every $j=1, \ldots, m$. \\
For the general case, we use induction in $\sharp\{i:p_i\not=\infty\}$. Let us suppose that the result is true for $\sharp\{i:p_i\not=\infty\}=k-1$ and let us prove it for $\sharp\{i:p_i\not=\infty\}=k$. We can suppose,
without loss of generality, that $p_1,\dots, p_k$ are all different from $\infty$ and so fix $n\in \mathbb{N}$ and consider $T\in\mathcal L(^m\ell_{p_1}^n, \ldots , \ell_{p_k}^n, \ell_{\infty}^n , \ldots , \ell_\infty^n;X)$. We write the $m$-linear mapping as
\[
T=\sum_{i_{1}, \ldots , i_{m} =1}^{n} a_{i_{1}, \ldots , i_{m}} e_{i_{1}, \ldots , i_{m}}, \quad\text{ where } e_{i_{1}, \ldots , i_{m}}=e'_{i_1}\cdots e'_{i_m} \,. 
\]
%
For each $x\in B_{\ell_{p_k}^n}$ let $T^{(x)}\in\mathcal L(^m\ell_{p_1}^n , \ldots ,\ell_{p_{k-1}}^n , \ell_{\infty}^n , \ldots , \ell_\infty^n;X)$ be given by
\[
T^{(x)}=\sum_{i_{1}, \ldots , i_{m} =1}^{n} a_{i_{1}, \ldots , i_{m}} x_{i_k} e_{i_{1}, \ldots , i_{m}}.
\]
Clearly, $\Vert T \Vert = \sup\{\Vert T^{(x)} \Vert: x\in B_{\ell_{p_k}^n}\}$. We can apply the inductive hypothesis to $T^{(x)}$: denoting $\frac{1}{\lambda^{*}}=\frac{1}{r}-\left(\frac{1}{p_1}+\cdots +\frac{1}{p_{k-1}}\right)$, we know, for all $j=1,\dots,m$  and all  $x\in B_{\ell_{p_k}^n}$,
\begin{equation}\label{hipotesis_inductiva}
\left(\sum_{i_j}\left(\sum_{[i_j]}\|v a_{i_{1}, \ldots , i_{m}}\|^q \, |x_{i_{k}}|^q\right)^{\frac{\lambda^{*}}{q}}\right)^{1/\lambda^{*}} \le K  \|T^{(x)}\|\le K  \|T\|.
\end{equation}
First of all, if  $j=k$ then we have, by the induction hypothesis,
\begin{multline*}
\left(\sum_{i_{k}}\left(\sum_{[i_{k}]}\|va_{i_{1}, \ldots , i_{m}}\|^q\right)^{\lambda/q}\right)^{1/\lambda}  =  \left(\sum_{i_{k}}\left(\sum_{[i_{k}]}\|va_{i_{1}, \ldots , i_{m}}\|^q\right)^{\frac{\lambda^{*}}{q} \cdot  \frac{\lambda}{\lambda^{*}}}\right)^{1/\lambda}\\
 = \sup_{x\in B_{\ell_{p_k}^n}} \left(\sum_{i_{k}}\left(\sum_{[i_{k}]}\|va_{i_{1}, \ldots , i_{m}}\|^q |x_{i_{k}}|^q\right)^{\lambda^{*}}\right)^{1/\lambda^{*}} \le K\|T\| \,.
\end{multline*}
Let us suppose now $j \not= k$. We denote $S_j=\left(\sum_{[i_{j}]}\|va_{i_{1}, \ldots , i_{m}}\|^q\right)^{1/q}$. Since $\lambda^{*}<\lambda<q$, some simple algebraic manipulations and the repeated use of H\"{o}lder's inequality yield
\begin{align*}
\sum_{i_{j}} &\left(\sum_{[i_{j}]}\|va_{i_{1}, \ldots , i_{m}}\|^q\right)^{\lambda/q} = \sum_{i_{j}} S_j^\lambda = \sum_{i_{j}} S_j^{\lambda-q} S_j^q= \sum_{i_{j}}\sum_{[i_{j}]}\frac{\|va_{i_{1}, \ldots , i_{m}}\|^q}{S_j^{q-\lambda}}
 \\&  = \sum_{i_{k}}\sum_{[i_{k}]}\frac{\|va_{i_{1}, \ldots , i_{m}}\|^q}{S_j^{q-\lambda}}
= \sum_{i_{k}}\sum_{[i_{k}]}\frac{\|va_{i_{1}, \ldots , i_{m}}\|^{\frac{q(q-\lambda)}{q-\lambda^{*}}}}{S_j^{q-\lambda}}  \|va_{i_{1}, \ldots , i_{m}}\|^{\frac{q(\lambda-\lambda^{*})}{q-\lambda^{*}}}\\
& \le  \sum_{i_{k}} \left(\sum_{[i_{k}]}\frac{\|va_{i_{1}, \ldots , i_{m}}\|^q}{S_j^{q-\lambda^{*}}}\right)^{\frac{q-\lambda}{q-\lambda^{*}} } \left(\sum_{[i_{k}]} \|va_{i_{1}, \ldots , i_{m}}\|^q\right)^{\frac{\lambda-\lambda^{*}}{q-\lambda^{*}}}\\
& \le  \left[ \sum_{i_{k}} \left(\sum_{[i_{k}]}\frac{\|va_{i_{1}, \ldots , i_{m}}\|^q}{S_j^{q-\lambda^{*}}}\right)^{\lambda/\lambda^{*}} \right]^{\frac{(q-\lambda)\lambda^{*}}{(q-\lambda^{*})\lambda}}
\left[ \sum_{i_{k}} \left(\sum_{[i_{k}]} \|va_{i_{1}, \ldots , i_{m}}\|^q\right)^{\lambda/q} \right]^{\frac{(\lambda-\lambda^{*})q}{(q-\lambda^{*})\lambda}} \,.
\end{align*}
We have already seen when proving the case $j=k$ that the second factor of the last product is bounded by $(K\|T\|)^{\frac{(\lambda-\lambda^{*})q}{q-\lambda^{*}}}$. Now we bound the first factor.
\begin{align*}
\bigg[ \sum_{i_{k}} \Big( &\sum_{[i_{k}]} \frac{\|va_{i_{1}, \ldots , i_{m}}\|^q}{S_j^{q-\lambda^{*}}} \Big)^{\lambda/\lambda^{*}} \bigg]^{\lambda^{*}/\lambda}
=  \sup_{x\in B_{\ell_{p_k}^n}} \sum_{i_{k}}\sum_{[i_{k}]} \frac{\|va_{i_{1}, \ldots , i_{m}}\|^q}{S_j^{q-\lambda^{*}}} |x_{i_{k}}|^{\lambda^{*}}\\
& =  \sup_{x\in B_{\ell_{p_k}^n}} \sum_{i_{j}} \sum_{[i_{j}]} \frac{\|va_{i_{1}, \ldots , i_{m}}\|^{q-\lambda^{*}}}{S_j^{q-\lambda^{*}}}  \|va_{i_{1}, \ldots , i_{m}}\|^{\lambda^{*}}|x_{i_{k}}|^{\lambda^{*}}\\
&\le \sup_{x\in B_{\ell_{p_k}^n}} \sum_{i_{j}}\left(\sum_{[i_{j}]}\frac{\|va_{i_{1}, \ldots , i_{m}}\|^{q}}{S_j^{q}} \right)^{\frac{q-\lambda^{*}}{q}}  \left(\sum_{[i_{j}]} \|va_{i_{1}, \ldots , i_{m}}\|^{q}|x_{i_{k}}|^{q} \right)^{\lambda^{*}/q}\\
& = \sup_{x\in B_{\ell_{p_k}^n}}  \sum_{i_{j}} \left(\sum_{[i_{j}]} \|va_{i_{1}, \ldots , i_{m}}\|^{q}|x_{i_{k}}|^{q} \right)^{\lambda^{*}/q}\\
&\le  (K\|T\|)^{\lambda^{*}}.
\end{align*}
Since the $n$ was arbitrary, this holds for every $n$ and completes the proof.
\end{proof}

\noindent We can now address the proof of our result.
\begin{proof}[Proof of Theorem~\ref{main main}] Let us assume first that $\lambda \geq q$. We proceed by induction on $m$. For $m=1$ we have $\frac{1}{\lambda} = \frac{1}{r} - \frac{1}{p_{1}}$. Since $v$ is $(r,1)$--summing,
then, by the Inclusion Theorem, it is also $(\lambda,p_{1}')$--summing. By \eqref{eq r1} this gives, for every operator $T:\ell_{p_{1}} \to X$ and every $n$
\[
 \Big( \sum_{j=1}^{n} \Vert v T(e_{j}) \Vert^{\lambda} \Big)^{1/\lambda} \leq \pi_{r,1}(v) \Vert T \Vert \, .
\]
For the inductive step we have $\frac{1}{\lambda} = \frac{1}{r} - \big(\frac{1}{p_{1}} + \cdots + \frac{1}{p_{m}} \big)$ and we consider the exponent
\[
\frac{1}{\lambda^{*}} = \frac{1}{r} - \big(\frac{1}{p_{2}} + \cdots + \frac{1}{p_{m}} \big) \, .
\]
We have now two possibilities, either $\lambda^{*} < q$ or $\lambda^{*} \geq q$. In the first case, given $T:\ell_{p_{1}}^{n} \times \ell_{p_{2}}^{n} \times \cdots \ell_{p_{m}}^{m} \to X$ with coefficients $(a_{i_1,\dots,i_m})$ we define
$\widetilde{T} : \ell_{\infty}^{n} \times \ell_{p_{2}}^{n} \times \cdots \ell_{p_{m}}^{m} \to X$ in the same way as $T$. Since $\frac{1}{\lambda^{*}} = \frac{1}{r} - \big(\frac{1}{\infty} + \frac{1}{p_{2}} + \cdots + \frac{1}{p_{m}} \big) > \frac{1}{q}$
we have, by Proposition~\ref{main tool}
\[
\bigg( \sum_{i_{1}}  \Big( \sum_{[i_{1}]} \Vert v \widetilde{T} (e_{i_{1}},e_{i_{2}} , \ldots , e_{i_{m}} ) \Vert^{q}  \Big)^{\frac{\lambda^{*}}{q}} \bigg)^{\frac{1}{\lambda^{*}}}
\leq  K
\|\widetilde T\|,
\]
where $K= \left(\sqrt{2}C_q(Y)\right)^{m-1} \pi_{r,1}(v)$. This, by \eqref{eq r1} means that the linear mapping $\mathcal{L} (^{m-1} \ell_{p_{2}}^{n} \times \cdots \ell_{p_{m}}^{m}; X) \to \ell_{q}^{n^{m-1}}(Y)$
given by $A \rightsquigarrow \big( v A(e_{i_{2}}, \ldots , e_{i_{m}} ) \big)_{i_{2}, \ldots , i_{m}}$ is $(\lambda^{*},1)$--summing. By the Inclusion Theorem this mapping is also $(\lambda,p'_{1})$--summing,
which means, again by \eqref{eq r1}
\[
\bigg( \sum_{i_{1}}  \Big( \sum_{[i_{1}]} \Vert v a_{i_{1}, \ldots , i_{m}} \Vert^{q}  \Big)^{\lambda/q} \bigg)^{1/\lambda}
\leq  K  \sup_{y^{(j)} \in B_{\ell_{p_{j}}^{n}} } \Big\Vert \sum_{i_{1}, \ldots , i_{m}} a_{i_{1}, \ldots , i_{m}} y^{(1)}_{i_{1}} y^{(2)}_{i_{2}} \cdots y^{(m)}_{i_{m}} \Big\Vert \, =K  \|T\|.
\]
Finally, since $\lambda \geq q$ we have
\[
   \Big( \sum_{i_{2}, \ldots , i_{m}}  \Vert v a_{i_{1}, \ldots , i_{m}} \Vert^{q}  \Big)^{1/q}
\geq  \Big( \sum_{i_{2}, \ldots , i_{m}}  \Vert v a_{i_{1}, \ldots , i_{m}} \Vert^{\lambda}  \Big)^{1/\lambda} \, .
\]
This completes the proof for this case.\\
Now, if $\lambda^{*} \geq q$ we have that $Y$ has cotype $\lambda^{*}$ and so also has $\ell_{\lambda^{*}}(Y)$. Then $\id : \ell_{\lambda^{*}}(Y) \to \ell_{\lambda^{*}}(Y)$
is $(\lambda^{*}, 1)$--summing and, by the ideal property (recall that $\lambda \geq \lambda^{*}$) \cite[Proposition~10.2]{DiJaTo95} $\id :  \ell_{\lambda^{*}}(Y)  \hookrightarrow \ell_{\lambda}(Y)$ is also
$(\lambda^{*}, 1)$--summing. Then the Inclusion Theorem gives $\pi_{\lambda, p'_{1}} (\id:  \ell_{\lambda^{*}}^{n^{m-1}}(Y) \hookrightarrow \ell_{\lambda}^{n^{m-1}}(Y) ) \leq C$ for every $n$ and $m$.
This means that for every $\big(b_{i_{2}, \ldots, i_{m}}^{(k)} \big)_{i_{2}, \ldots, i_{m}=1}^{n} \subseteq  \ell_{\lambda}^{n^{m-1}}(Y)$, with $k=1, \ldots , N$
\[
  \Big( \sum_{k=1}^{N} \big\Vert \big(b_{i_{2}, \ldots, i_{m}}^{(k)} \big)_{i_{2}, \ldots, i_{m}} \big\Vert_{ \ell_{\lambda}(Y)}^{\lambda}  \Big)^{1/\lambda}
\leq C \sup_{\gamma \in B_{\ell_{\lambda^{*}}^{n}(Y)^{*}}} \Big(  \sum_{k=1}^{N} \big\vert \gamma \big(b^{(k)} \big)  \big\vert^{p_{1}'}  \Big)^{1/p_{1}'} \,.
\]
Then, if $T \in \mathcal{L}(^{m} \ell_{p_{1} }, \ldots , \ell_{p_{m}} ; X ) $ with coefficients $(a_{i_{1}, \ldots , i_{m}})$ we write $b_{i_{2}, \ldots, i_{m}}^{(i_{1})} =  v a_{i_{1}, \ldots , i_{m}}$
and we have
\begin{multline*}
  \Big( \sum_{i_{1}, \ldots , i_{m} } \Vert v a_{i_{1}, \ldots , i_{m}} \Vert^{\lambda}  \Big)^{1/\lambda}  = \Big(\sum_{i_{1}} \Vert b^{(i_{1})} \Vert_{ \ell_{\lambda}(Y)}  ^\lambda \Big)^{1/\lambda}
 \leq C  \sup_{\gamma \in B_{\ell_{\lambda^{*}}^{n}(Y)^{*}}} \Big(  \sum_{i_{1}=1}^{n} \big\vert \gamma \big(b^{(i_{1})} \big)  \big\vert^{p_{1}'}  \Big)^{1/p_{1}'} \\
 =  C  \sup_{\gamma \in B_{\ell_{\lambda^{*}}^{n}(Y)^{*}}}  \sup_{x \in B_{\ell_{p_{1}}^{n}}}  \Big\vert  \sum_{i_{1}=1}^{n}   \gamma \big(b^{(i_{1})} \big) x_{i_{1}} \Big\vert
= C   \sup_{x \in B_{\ell_{p_{1}}^{n}}}  \sup_{\gamma \in B_{\ell_{\lambda^{*}}^{n}(Y)^{*}}} \Big\vert  \gamma \Big( \sum_{i_{1}=1}^{n}   b^{(i_{1})} x_{i_{1}}  \Big) \Big\vert  \\
= C   \sup_{x \in B_{\ell_{p_{1}}^{n}}} \Big\Vert   \sum_{i_{1}=1}^{n}    v a_{i_{1}, \ldots , i_{m}}  x_{i_{1}}   \Big\Vert_{\ell_{\lambda^{*}}^{n}(Y)}
= C   \sup_{x \in B_{\ell_{p_{1}}^{n}}} \Big( \sum_{i_{2}, \ldots , i_{m} } \Big\Vert v\big(  \sum_{i_{1}=1}^{n}  a_{i_{1}, \ldots , i_{m}}  x_{i_{1}}  \big)   \Big\Vert^{\lambda^{*}} \Big)^{\frac{1}{\lambda^{*}}}  \, .
\end{multline*}
We now apply the induction hypothesis with the $(m-1)$--mapping whose coefficients are $\big(  \sum_{i_{1}=1}^{n}  a_{i_{1}, \ldots , i_{m}}  x_{i_{1}}  \big)_{i_{2}, \ldots , i_{m}}$ to have
\[
   \Big( \sum_{i_{2}, \ldots , i_{m} } \Big\Vert v\big(  \sum_{i_{1}=1}^{n}  a_{i_{1}, \ldots , i_{m}}  x_{i_{1}}  \big)   \Big\Vert^{\lambda^{*}} \Big)^{1/\lambda^{*}}
\leq K \sup_{y^{(j)} \in B_{\ell_{p_{j}}^{n}} } \Big\Vert  \sum_{i_{2}, \ldots , i_{m} }  \sum_{i_{1}=1}^{n}  a_{i_{1}, \ldots , i_{m}}  x_{i_{1}} y^{(2)}_{i_{2}} \cdots y^{(m)}_{i_{m}}   \Big\Vert_{X} \, .
\]
This completes the proof of \eqref{main-i}.\\
We prove now  \eqref{main-ii}. If $m=1$ we have $\mu=\lambda$ and then it follows as in the previous case. For a general $m$ let us first note that
the statement can be refrased in terms of tensor products as
\[
 \sup_{n} \left\Vert \id \otimes v: \ell_{p_1'}^n\otimes_{\varepsilon} \cdots \otimes_{\varepsilon} \ell_{p_m'}^n\otimes_{\varepsilon} X \to \ell_\mu^{n^m}(Y)\right\Vert \leq K \, .
\]
We are going to iterate a procedure of  intertwining, transposition and interpolation. First observe that $\lambda <q$ gives $\frac{1}{p_1}+\cdots +\frac{1}{p_m}<\frac{1}{r}-\frac{1}{q}$ and
then, by Proposition~\ref{main tool} we have (denoting $K=\left(\sqrt{2}C_q(Y)\right)^{m-1} \pi_{r,1}(v)$)
\begin{equation}\label{interpol_1}
\sup_{n} \left \Vert \id \otimes v:\ell_{p_1'}^n\otimes_{\varepsilon}\ell_{p_2'}^n\otimes_{\varepsilon}\cdots\otimes_{\varepsilon} \ell_{p_m'}^n\otimes_{\varepsilon} X
\to \ell_\lambda^n\big(\ell_q^{n^{m-1}}(Y)\big)\right\Vert \leq K \, ,
\end{equation}
and also
\[
\sup_{n} \left \Vert \id \otimes v:\ell_{p_2'}^n\otimes_{\varepsilon}\ell_{p_1'}^n\otimes_{\varepsilon} \ell_{p_3'}^n\otimes_{\varepsilon}\cdots\otimes_{\varepsilon} \ell_{p_m'}^n\otimes_{\varepsilon} X
\to \ell_\lambda^n \big(\ell_q^{n^{m-1}}(Y)\big) \right\Vert \leq K \, .
\]
We fix now $n$; by Minkowski's inequality (recall that $\lambda <q$), the transposition operator $\tau:\ell_\lambda^n\big(\ell_q^{n}(Y)\big)\to \ell_q^n\big(\ell_\lambda^n(Y)\big)$ has norm $1$.
The intertwining operator given by
\begin{eqnarray*}
\rho_2: \ell_{p_1'}^n\otimes_{\varepsilon}\ell_{p_2'}^n &\to & \ell_{p_2'}^n\otimes_{\varepsilon}\ell_{p_1'}^n\\
a\otimes b &\mapsto & b\otimes a
\end{eqnarray*}
also has norm $1$.\\
So we have the following three operators:
\begin{itemize}
\item $\rho_2\otimes \id:\ell_{p_1'}^n\otimes_{\varepsilon}\ell_{p_2'}^n\otimes_{\varepsilon}\cdots\otimes_{\varepsilon} \ell_{p_m'}^n\otimes_{\varepsilon} X\to \ell_{p_2'}^n\otimes_{\varepsilon}\ell_{p_1'}^n\otimes_{\varepsilon} \ell_{p_3'}^n\otimes_{\varepsilon}\cdots\otimes_{\varepsilon} \ell_{p_m'}^n\otimes_{\varepsilon} X$,

\item $\id\otimes v: \ell_{p_2'}^n\otimes_{\varepsilon}\ell_{p_1'}^n\otimes_{\varepsilon} \ell_{p_3'}^n\otimes_{\varepsilon}\cdots\otimes_{\varepsilon} \ell_{p_m'}^n\otimes_{\varepsilon} X\to \ell_\lambda^n\big(\ell_q^{n^{m-1}}(Y)\big)$,

\item $\tau: \ell_\lambda^n\big(\ell_q^{n^{m-1}}(Y)\big)\to \ell_q^n\left( \ell_\lambda^n\big(\ell_q^{n^{m-2}}(Y)\big)\right)$,
\end{itemize}
Composing them we have
\begin{equation}\label{interpol_2}
\left\|\id\otimes
 v:\ell_{p_1'}^n\otimes_{\varepsilon}\ell_{p_2'}^n\otimes_{\varepsilon}\cdots\otimes_{\varepsilon} \ell_{p_m'}^n\otimes_{\varepsilon} X\to \ell_q^n\left( \ell_\lambda^n\big(\ell_q^{n^{m-2}}(Y)\big)\right)\right\| \le K.
 \end{equation}
We now use complex interpolation of \eqref{interpol_1} and \eqref{interpol_2} with $\theta = 1/2$ (see e.g. \cite[Chapter~3]{BeLo76}) to get
\[
\left\Vert \id\otimes v:\ell_{p_1'}^n\otimes_{\varepsilon}\ell_{p_2'}^n\otimes_{\varepsilon}\cdots\otimes_{\varepsilon} \ell_{p_m'}^n\otimes_{\varepsilon} X
\to \ell_{\mu_2}^{n^2}\big(\ell_q^{n^{m-2}}(Y)\big)\right\Vert \leq K \, ,
\]
where $\frac{1}{\mu_2}=\frac{\frac{1}{2}}{\lambda}+\frac{\frac{1}{2}}{q}$.\\
Now, since $\mu_2<q$, again we have that the first and third of the following mappings (defined in the obvious way) have norm $1$, and the norm of the second one is bounded by $K$:
\begin{itemize}
\item $\rho_3\otimes \id:\ell_{p_1'}^n\otimes_{\varepsilon}\ell_{p_2'}^n\otimes_{\varepsilon}\cdots\otimes_{\varepsilon} \ell_{p_m'}^n\otimes_{\varepsilon} X\to \ell_{p_2'}^n\otimes_{\varepsilon}\ell_{p_3'}^n\otimes_{\varepsilon} \ell_{p_1'}^n\otimes_{\varepsilon}\cdots\otimes_{\varepsilon} \ell_{p_m'}^n\otimes_{\varepsilon} X$,

\item $\id\otimes v: \ell_{p_2'}^n\otimes_{\varepsilon}\ell_{p_3'}^n\otimes_{\varepsilon} \ell_{p_1'}^n\otimes_{\varepsilon}\cdots\otimes_{\varepsilon} \ell_{p_m'}^n\otimes_{\varepsilon} X\to \ell_{\mu_2}^{n^2}\big(\ell_q^{n^{m-2}}(Y)\big)$,

\item $\tau: \ell_{\mu_2}^{n^2}\big(\ell_q^{n^{m-2}}(Y)\big)\to \ell_q^n\left( \ell_{\mu_2}^{n^2}\big(\ell_q^{n^{m-3}}(Y)\big)\right)$.
\end{itemize}
We compose these three mappings to obtain
\begin{equation}\label{interpol_3}
\left\Vert \id\otimes  v:\ell_{p_1'}^n\otimes_{\varepsilon}\ell_{p_2'}^n\otimes_{\varepsilon}\cdots\otimes_{\varepsilon} \ell_{p_m'}^n\otimes_{\varepsilon} X
\to \ell_q^n\left( \ell_{\mu_2}^{n^2}\big(\ell_q^{n^{m-3}}(Y)\big)\right)\right\Vert \leq K.
\end{equation}
We again interpolate \eqref{interpol_1} and \eqref{interpol_3} with the complex method and $\theta=1/3$,
\[
\left\Vert \id\otimes v:\ell_{p_1'}^n\otimes_{\varepsilon}\ell_{p_2'}^n\otimes_{\varepsilon}\cdots\otimes_{\varepsilon} \ell_{p_m'}^n\otimes_{\varepsilon} X
\to \ell_{\mu_3}^{n^3}\big(\ell_q^{n^{m-3}}(Y)\big)\right\Vert \leq K \, ,
\]
where $\frac{1}{\mu_3}=\frac{\frac{1}{3}}{\lambda}+\frac{\frac{2}{3}}{q}= \frac{\frac{1}{3}}{q}+\frac{\frac{2}{3}}{\mu_2} $.\\
Following the same procedure we finally end up in
\[
\left\Vert \id\otimes v:\ell_{p_1'}^n\otimes_{\varepsilon}\ell_{p_2'}^n\otimes_{\varepsilon}\cdots\otimes_{\varepsilon} \ell_{p_m'}^n\otimes_{\varepsilon} X
\to \ell_{\mu_m}^{n^m}(Y)\right\Vert \leq K \, ,
\]
where $\frac{1}{\mu_m}=\frac{\frac{1}{m}}{\lambda}+\frac{\frac{m-1}{m}}{q}$.
\end{proof}

\section{Some consequences}

We present now some results that follow immediately from Theorem~\ref{main main}. The first one is for scalar valued multilinear mappings and completes the result in
\cite{PP81} with the cases that were not considered there. We also show that the exponents are optimal.
\begin{proposition} \label{praciano}
  Let $1 \leq p_{1}, \ldots , p_{m} \leq \infty$ such that $\frac{1}{p_{1}} + \cdots + \frac{1}{p_{m}} < 1$. Consider the exponents
\[
  \frac{1}{\lambda} =  1 - \Big(\frac{1}{p_{1}} + \cdots + \frac{1}{p_{m}} \Big) \qquad \text{ and } \qquad
\frac{1}{\mu} =  \frac{1}{m \lambda} + \frac{m-1}{2m} \,.
\]
Then there exists $C>0$ such that, for every $m$-linear $T : \ell_{p_{1}} \times \cdots \times \ell_{p_{m}} \to \mathbb{C}$
with coefficients $(a_{i_{1}, \ldots , i_{m}})$ we have
\begin{enumerate}
\item\label{praciano-i} If $\frac{1}{2} \leq \frac{1}{p_{1}} + \cdots + \frac{1}{p_{m}} < 1$, then $ \displaystyle
\Big( \sum_{i_{1}, \ldots, i_{m}=1}^{\infty} \vert a_{i_{1}, \ldots , i_{m}} \vert^{\lambda}\Big)^{1/\lambda} \leq C \Vert T \Vert$.

\item\label{praciano-ii} If $0 \leq \frac{1}{p_{1}} + \cdots + \frac{1}{p_{m}} < \frac{1}{2}$, then $ \displaystyle
\Big( \sum_{i_{1}, \ldots, i_{m}=1}^{\infty} \vert a_{i_{1}, \ldots , i_{m}} \vert^{\mu}\Big)^{1/\mu} \leq C \Vert T \Vert$.
\end{enumerate}
Moreover the exponents are optimal.
\end{proposition}
\begin{proof}
The inequalities follow from Theorem~\ref{main main} using that $\mathbb{C}$ has cotype $2$ and that the identity on $\mathbb{C}$ is $(1,1)$--summing. Let us assume now that
$t$ is such that for every $T: \ell_{p_{1}} \times \cdots \times \ell_{p_{m}} \to \mathbb{C}$ with $\frac{1}{2} \leq \frac{1}{p_{1}} + \cdots + \frac{1}{p_{m}} < 1$ we have
\begin{equation} \label{t}
  \Big( \sum_{i_{1} , \ldots , i_{m}} \vert a_{i_{1} , \ldots , i_{m}} \vert^{t} \Big)^{1/t} \leq C \Vert T \Vert
\end{equation}
for some universal $C>0$. Define $\Phi_{n} : \ell_{p_{1}}^{n} \times \cdots \times \ell_{p_{m}}^{n} \to \mathbb{C}$ by $\Phi_{n} (x^{(1)}, \ldots , x^{(m)} ) = \sum_{i=1}^{n} x^{(1)}_{i} \cdots  x^{(m)}_{i}$.
Using H\"older's inequality it is easily seen that $\Vert \Phi_{n} \Vert \leq n^{1/\lambda}$. Then, if \eqref{t} holds then we have $n^{1/t} \leq C n^{1/\lambda}$ for every $n$, which gives $t \geq \lambda$.\\
For \eqref{praciano-ii} let us note first that the condition $0 \leq \frac{1}{p_{1}} + \cdots + \frac{1}{p_{m}} < \frac{1}{2}$ implies $p_{j} > 2$ for every $j=1, \ldots , m$. We show first that
for $p_{1} , \ldots , p_{m} > 2$ there
is a constant $K_{m} >0$ such that if $\big( g_{i_{1} , \ldots , i_{m}} \big)_{i_{1} , \ldots , i_{m}=1}^{n}$ are independent Gaussian random variables we have
\begin{equation} \label{gaussian}
  \int \Big\Vert \sum_{i_{1} , \ldots , i_{m}=1}^{n} g_{i_{1} , \ldots , i_{m}} (\omega) e_{i_{1}} \otimes \cdots \otimes e_{i_{m}} \Big\Vert_{\ell_{p'_{1}}^{n} \otimes_{\varepsilon} \cdots \otimes_{\varepsilon} \ell_{p'_{m}}^{n} } d \omega
\leq K_{m} n^{\frac{1}{\lambda} + \frac{m-1}{2}} \,.
\end{equation}
We proceed by induction. It is well known (see e.g. \cite[(4)]{DeMaSe02}) that for $m=1$ there is $K_{1}>0$ such that $\int \big\Vert \sum_{i=1}^{n} g_{i} (\omega) e_{i} \big\Vert_{\ell_{p'_{1}}^{n}} d \omega
\leq K_{1} n^{1/p'_{1}}$. We assume that \eqref{gaussian} holds for an $(m-1)$-fold tensor product and take families of independent Gaussian random variables $\big( g_{i_{1} , \ldots , i_{m-1}} \big)$ and $\big(g_{k} \big)$.
By Chev\'et's inequality (see \cite[(43.2)]{TJ89}) there is a constant $C>0$ such that
\begin{multline*}
   \int \Big\Vert \sum_{i_{1} , \ldots , i_{m}=1}^{n} g_{i_{1} , \ldots , i_{m}} (\omega) e_{i_{1}} \otimes \cdots \otimes e_{i_{m}} \Big\Vert_{\ell_{p'_{1}}^{n} \otimes_{\varepsilon} \cdots \otimes_{\varepsilon} \ell_{p'_{m}}^{n} } d \omega \\
\leq C \Big( \big\Vert \id: \ell_{2}^{n} \hookrightarrow \ell_{p'_{m}}^{n} \big\Vert \, \,   \int \Big\Vert \sum_{i_{1} , \ldots , i_{m-1}=1}^{n} g_{i_{1} , \ldots , i_{m-1}} (\omega) e_{i_{1}} \otimes \cdots \otimes e_{i_{m-1}} \Big\Vert_{\ell_{p'_{1}}^{n} \otimes_{\varepsilon} \cdots \otimes_{\varepsilon} \ell_{p'_{m-1}}^{n} } d \omega \\
+ \big\Vert \id : \ell_{2}^{n^{m-1}} \hookrightarrow  \ell_{p'_{1}}^{n} \otimes_{\varepsilon} \cdots \otimes_{\varepsilon} \ell_{p'_{m-1}}^{n} \big\Vert  \,\,
\int \Big\Vert \sum_{k=1}^{n} g_{k} (\omega) e_{k} \Big\Vert_{\ell_{p'_{m}}^{n}} d \omega   \Big) \, .
\end{multline*}
By the metric mapping property of $\varepsilon$ we have
\begin{multline*}
  \big\Vert \id : \ell_{2}^{n^{m-1}} \hookrightarrow  \ell_{p'_{1}}^{n} \otimes_{\varepsilon} \cdots \otimes_{\varepsilon} \ell_{p'_{m-1}}^{n} \big\Vert
\leq \prod_{i=1}^{m-1} \big\Vert \id: \ell_{2}^{n} \hookrightarrow \ell_{p'_{i}}^{n} \big\Vert \\
= \prod_{i=1}^{m-1} n^{\frac{1}{p'_{i}} - \frac{1}{2}} = n^{\sum_{i=1}^{m-1} \frac{1}{2} - \frac{1}{p_{i} } } = n^{\frac{m-1}{2} - \sum_{i=1}^{m-1} \frac{1}{p_{i} } }
\end{multline*}
With this,  the induction hypothesis  and the case $m=1$, we have
\begin{multline*}
 \int \Big\Vert \sum_{i_{1} , \ldots , i_{m}=1}^{n} g_{i_{1} , \ldots , i_{m}} (\omega) e_{i_{1}} \otimes \cdots \otimes e_{i_{m}} \Big\Vert_{\ell_{p'_{1}}^{n} \otimes_{\varepsilon} \cdots \otimes_{\varepsilon} \ell_{p'_{m}}^{n} } d \omega \\
\leq C \Big( n^{\frac{1}{p'_{m}} - \frac{1}{2}} K_{m-1} n^{\frac{1}{\lambda^{*}} + \frac{m-2}{2}} + n^{\frac{m-1}{2} - \sum_{i=1}^{m-1} \frac{1}{p_{i} } } K_{1} n^{\frac{1}{p'_{m}} }   \Big) \, ,
\end{multline*}
where $\frac{1}{\lambda^{*}} = 1 - \big( \frac{1}{p_{1} } + \cdots + \frac{1}{p_{m-1} } \big) $. Noting that $\frac{1}{p'_{m}} - \frac{1}{2} + \frac{1}{\lambda^{*}} + \frac{m-2}{2} =
\frac{m-1}{2} - \sum_{i=1}^{m-1} \frac{1}{p_{i} } +\frac{1}{p'_{m}} = \frac{1}{\lambda} + \frac{m-1}{2}$ we finally have \eqref{gaussian}.\\
It is a well known fact that Bernoulli averages are dominated by Gaussian averages \cite[Proposition~12.11]{DiJaTo95}, then there is a constant $K>0$ such that for all $n$
\[
    \int \Big\Vert \sum_{i_{1} , \ldots , i_{m}=1}^{n} \varepsilon_{i_{1} , \ldots , i_{m}} (\omega) e_{i_{1}} \otimes \cdots \otimes e_{i_{m}} \Vert_{\ell_{p'_{1}}^{n} \otimes_{\varepsilon} \cdots \otimes_{\varepsilon} \ell_{p'_{m}}^{n} } d \omega
\leq K n^{\frac{1}{\lambda} + \frac{m-1}{2}} \,.
\]
Then for each $n$ there is a choice of signs $\varepsilon_{i_{1} , \ldots , i_{m}} = \pm 1$ such that $z= \sum_{i_{1} , \ldots , i_{m}} \varepsilon_{i_{1} , \ldots , i_{m}} e_{i_{1}} \otimes \cdots \otimes e_{i_{m}}$
satisfies $\Vert z \Vert_{\ell_{p'_{1}}^{n} \otimes_{\varepsilon} \cdots \otimes_{\varepsilon} \ell_{p'_{m}}^{n} } \leq K n^{\frac{1}{\lambda} + \frac{m-1}{2}}$. Since
$\big( \sum_{i_{1} , \ldots , i_{m}=1}^{n} \vert \varepsilon_{i_{1} , \ldots , i_{m}} \vert^{t} \big)^{1/t} = n^{m/t}$, if \eqref{t} holds for $p_{1}, \ldots , p_{m}$ satisfying \eqref{praciano-ii} we
have $n^{m/t} \leq K n^{\frac{1}{\lambda} + \frac{m-1}{2}}$, which implies $t \geq \mu$.
\end{proof}

\begin{remark}\label{necesario}
 The condition $\frac{1}{p_{1}} + \cdots + \frac{1}{p_{m}} < 1$ is necessary in Proposition~\ref{praciano}. Indeed, if $\frac{1}{p_{1}} + \cdots + \frac{1}{p_{m}} \geq 1$ then the mapping
$\Phi : \ell_{p_{1}} \times \cdots \times  \ell_{p_{m}} \to \mathbb{C}$ given by $\Phi_{n} (x^{(1)}, \ldots , x^{(m)} ) = \sum_{i=1}^{\infty} x^{(1)}_{i} \cdots  x^{(m)}_{i}$ is well defined
and has infinitely many coefficients equal to $1$. Hence, there is no exponent $t$ satisfying an inequality like in Proposition~\ref{praciano}.
\end{remark}

\noindent If $X$ is a Banach space with cotype $q$ then the identity is $(q,1)$--summing and we obtain from Theorem~\ref{main main}
\begin{proposition}\label{cotipo}
Let $2 \leq p_{1} , \ldots , p_{m} \leq \infty$ and $q\geq 2$ such that $\frac{1}{p_{1}} + \cdots + \frac{1}{p_{m}} < \frac{1}{q}$. Define
\[
  \frac{1}{\lambda} =  \frac{1}{q} - \Big( \frac{1}{p_{1}} + \cdots + \frac{1}{p_{m}} \Big) \, .
\]
 Then for each Banach space $X$ with cotype $q$  there exists $C>0$ such that for every continuous, $m$-linear $T: \ell_{p_{1}} \times \cdots \times  \ell_{p_{m}} \to X$
with coefficients $(a_{i_{1} , \ldots , i_{m} })$ we have
\[
\Big( \sum_{i_{1} , \ldots , i_{m} =1}^{\infty} \Vert a_{i_{1} , \ldots , i_{m} } \Vert_{X}^{\lambda} \Big)^{1/\lambda} \leq C \Vert T \Vert \, .
\]
\end{proposition}

\bigskip

\noindent We can now give the following result, from which Theorem~\ref{main polin} readily follows. Let us note that by \cite[Lemma~5]{DeSe09} the fact that an exponent is optimal in an inequality for $m$-linear mappings implies
that it is also optimal for the corresponding inequality for $m$-homogeneous polynomials. Hence, the optimality of the exponents in  Theorem~\ref{main polin} also follows.
\begin{proposition}\label{ele p}
  Let $1 \leq p_{1} , \ldots , p_{m} \leq \infty$ and $1\leq u \leq q \leq \infty$. Then there is $C>0$ such that, for every continuous $m$-linear
$T: \ell_{p_{1}} \times \cdots \times  \ell_{p_{m}} \to \ell_{u}$ with coefficients $(a_{i_{1} , \ldots , i_{m} })$ we have
\[
  \Big( \sum_{i_{1} , \ldots , i_{m}=1}^{\infty} \Vert a_{i_{1} , \ldots , i_{m} } \Vert_{\ell_{q}}^{\rho} \Big)^{1/\rho} \leq C \Vert T \Vert \, ,
\]
where $\rho$ is given by
\begin{enumerate}
  \item\label{ele p-i} If $1 \leq u \leq q \leq 2$, and
      \begin{enumerate}
        \item\label{ele p-i-a} if $0 \leq \frac{1}{p_{1}} + \cdots + \frac{1}{p_{m}} < \frac{1}{u} - \frac{1}{q}$, then $\rho = \frac{2m}{m+2(1/u - 1/q - (1/p_{1} + \cdots + 1/p_{m}))}$.
	\item\label{ele p-i-b} if $ \frac{1}{u} - \frac{1}{q} \leq \frac{1}{p_{1}} + \cdots + \frac{1}{p_{m}} < \frac{1}{2} + \frac{1}{u} - \frac{1}{q}$, then
						  $\rho = \frac{2}{1+2(1/u - 1/q - (1/p_{1} + \cdots + 1/p_{m}))}$.
      \end{enumerate}
  \item\label{ele p-ii} If  $1 \leq u \leq 2 \leq q$, and
      \begin{enumerate}
        \item\label{ele p-ii-a} if $0 \leq \frac{1}{p_{1}} + \cdots + \frac{1}{p_{m}} < \frac{1}{u} - \frac{1}{2}$, then $\rho = \frac{2m}{m+2(1/u - 1/2 - (1/p_{1} + \cdots + 1/p_{m}))}$.
	\item\label{ele p-ii-b} if $ \frac{1}{u} - \frac{1}{2} \leq \frac{1}{p_{1}} + \cdots + \frac{1}{p_{m}} < \frac{1}{u}$, then
						  $\rho = \frac{1}{1/u -  (1/p_{1} + \cdots + 1/p_{m})}$.
      \end{enumerate}
  \item\label{ele p-iii} If  $2 \leq u \leq q \leq \infty$ and $0 \leq \frac{1}{p_{1}} + \cdots + \frac{1}{p_{m}} < \frac{1}{u}$, then $\rho = \frac{1}{1/u -  (1/p_{1} + \cdots + 1/p_{m})}$.
\end{enumerate}
Moreover, the exponents in the cases \eqref{ele p-i-a}, \eqref{ele p-ii-b} and \eqref{ele p-iii} are optimal. Also, the exponent in \eqref{ele p-i-b} is optimal for  $\frac{1}{u} - \frac{1}{q} \le \frac{1}{p_{1}} + \cdots + \frac{1}{p_{m}} < \frac{1}{2} $.
\end{proposition}
\noindent Let us remark that, by doing $p_{1} = \ldots = p_{m} = \infty$ we again find the exponents in \cite[Theorem~1]{DeSe09}.
\begin{proof}
  The case \eqref{ele p-i} follows immediately from Theorem~\ref{main main}, taking $v=\id: \ell_{u} \hookrightarrow \ell_{q}$ that is $(r,1)$--summing with
$\frac{1}{r} = \frac{1}{2} + \frac{1}{u} - \frac{1}{q} $ (by the Bennett--Carl inequalities) and that $\ell_{q}$ has cotype $2$.\\
The case \eqref{ele p-ii} follows from the previous one with $\id: \ell_{u} \hookrightarrow \ell_{2}$ and the fact that $\Vert \, \Vert_{q} \leq \Vert \, \Vert_{2}$.\\
Finally, the case \eqref{ele p-iii} follows from Proposition~\ref{cotipo} (since $\ell_{u}$ has cotype $u$) and the fact that $\Vert \, \Vert_{q} \leq \Vert \, \Vert_{u}$.\\
To see that the exponent is optimal, let us suppose that $t\geq 1$ is such that for every $T \in \mathcal{L}(^{m} \ell_{p_{1}}, \ldots , \ell_{p_{m}}; \ell_{u} )$ we have
\begin{equation} \label{ecuacion t}
 \Big( \sum_{i_{1}, \ldots , i_{m}=1}^{n} \Vert a_{i_{1}, \ldots , i_{m}} \Vert_{\ell_q}^{t} \Big)^{1/t} \leq C \Vert T \Vert \, ,
\end{equation}
for some universal $C>0$. Equivalently,
\[
 \sup_{n} \big\Vert \id: \ell_{p'_{1}}^{n} \otimes_{\varepsilon} \cdots \otimes_{\varepsilon} \ell_{p'_{m}}^{n} \otimes_{\varepsilon} \ell_{u}^{n} \to \ell_{t}^{n^{m} } \big(\ell_{q}^{n} \big) \big\Vert \leq C \, .
\]
In \eqref{ele p-i-a} we can proceed as in \eqref{gaussian} (taking into account that we have $\frac{1}{p_{1}} + \cdots + \frac{1}{p_{m}} < \frac{1}{2}$ and $u'\geq 2$) to find a choice of signs $\varepsilon_{i_{1}, \ldots , i_{m+1}} = \pm 1$
such that $z= \sum_{i_{1}, \ldots , i_{m}} \varepsilon_{i_{1}, \ldots , i_{m+1}} e_{i_{1}} \otimes \cdots \otimes e_{i_{m}} \otimes e_{i_{m+1}}$ satisfies
\[
 \Vert z \Vert_{\ell_{p'_{1}}^{n} \otimes_{\varepsilon} \cdots \otimes_{\varepsilon} \ell_{p'_{m}}^{n} \otimes_{\varepsilon} \ell_{u}^{n}} \leq n^{1-(\frac{1}{p_{1}} + \cdots + \frac{1}{p_{m}} ) - \frac{1}{u'} + \frac{m}{2} } \,.
\]
On the other hand, proceeding as in \cite[Section~3.1]{DeSe09} we have $\Vert z \Vert_{\ell_{t}(\ell_{q})} = n^{m/t +1/q}$. Then, if \eqref{ecuacion t} holds, this implies
$ \frac{m}{t} + \frac{1}{q} \leq \frac{1}{u} -\big(\frac{1}{p_{1}} + \cdots + \frac{1}{p_{m}} \big)  + \frac{m}{2}$ \,,
which gives
\[
 \frac{1}{t}  \leq   \frac{1}{2} + \frac{1}{m} \Big( \frac{1}{u} - \frac{1}{q} - \big(\frac{1}{p_{1}} + \cdots + \frac{1}{p_{m}} \big) \Big) \,\textrm{ and so, } t\ge\rho.
\]
Now, if $\frac{1}{p_{1}} + \cdots + \frac{1}{p_{m}} < \frac{1}{u}$ we consider $T : \ell_{p_{1}}^{n} \times \cdots \times  \ell_{p_{m}}^{n} \to \ell_{u}^{n}$ given by $T(x^{(1)} , \ldots , x^{(m)} ) = \sum_{j=1}^{n}x^{(1)}_{j}  \cdots x^{(m)}_{j} e_{j}$.
Taking $x^{(i)} \in B_{\ell_{p_{i}}}$ for $i=1, \ldots, m$ we have
\begin{align*}
 \Vert T(x^{(1)} , \ldots , x^{(m)} ) \Vert_{\ell_{u}} =& \Big( \sum_{j=1}^{n} \vert x^{(1)}_{j}  \cdots x^{(m)}_{j} \vert^{u} \Big)^{1/u}
= \sup_{y \in B_{\ell_{u'}}} \Big\vert \sum_{j=1}^{n} \ x^{(1)}_{j}  \cdots x^{(m)}_{j} y_{j} \Big\vert \\
\leq &\Big( \sum_{j} \vert x^{(1)}_{j} \vert^{p_{1}} \Big)^{1/p_{1}} \cdots \Big( \sum_{j} \vert x^{(m)}_{j} \vert^{p_{1}} \Big)^{1/p_{m}}\\ & \sup_{y \in B_{\ell_{u'}}} \Big( \sum_{j} \vert y_{j} \vert^{u'} \Big)^{1/u'}
\Big( \sum_{j} 1 \Big)^{1 - \frac{1}{u'} - (\frac{1}{p_{1}} + \cdots + \frac{1}{p_{m}} ) }  \\
\leq & n^{ \frac{1}{u} - (\frac{1}{p_{1}} + \cdots + \frac{1}{p_{m}} )} \, .
\end{align*}
On the other hand,
$T(e_{i_{1}} , \ldots  e_{i_{m}}) = e_{i}$ if $i_{1} = \ldots i_{m} = i$ and the null vector, otherwise. Then $\big( \sum \Vert T(e_{i_{1}} , \ldots  e_{i_{m}}) \Vert_{\ell_q}^{t} \big)^{1/t}=n^{1/t}$ and, if \eqref{ecuacion t} holds we have
\[
 \frac{1}{t} \leq \frac{1}{u} - \big(\frac{1}{p_{1}} + \cdots + \frac{1}{p_{m}} \big) \,.
\]
Thus, $t\ge\rho$ in the cases \eqref{ele p-ii-b} and \eqref{ele p-iii}.

For $\frac{1}{u} - \frac{1}{q} \le \frac{1}{p_{1}} + \cdots + \frac{1}{p_{m}} < \frac{1}{2}$ (and $1 \leq u \leq q \leq 2$) we consider the Fourier $n \times n$ matrix $a_{kl}=e^{\frac{2\pi i k l}{n}}$
and define $T : \ell_{p_{1}}^{n} \times \cdots \times  \ell_{p_{m}}^{n} \to \ell_{u}^{n}$ by $T(x^{(1)} , \ldots , x^{(m)} ) = \sum_{i=1}^{n} \sum_{j=1}^{n} a_{ij} x^{(1)}_{j}  \cdots x^{(m)}_{j} e_{i}$.
For $x^{(i)} \in B_{\ell_{p_{i}}}$ with $i=1, \ldots, m$ we have
\begin{align*}
 \Vert T(x^{(1)} , \ldots , x^{(m)} ) \Vert_{\ell_{u}} =& \Big( \sum_{i=1}^{n} \Big\vert \sum_{j=1}^{n} a_{ij} x^{(1)}_{j}  \cdots x^{(m)}_{j} \Big\vert^{u} \Big)^{1/u}
= \sup_{y \in B_{\ell_{u'}}} \Big\vert \sum_{i,j=1}^{n} a_{ij} x^{(1)}_{j}  \cdots x^{(m)}_{j} y_{i} \Big\vert \\
\leq & \Big( \sum_{j} \vert x^{(1)}_{j} \vert^{p_{1}} \Big)^{1/p_{1}} \cdots \Big( \sum_{j} \vert x^{(m)}_{j} \vert^{p_{1}} \Big)^{1/p_{m}} \sup_{y \in B_{\ell_{u'}}} \Big( \sum_{j} \big\vert \sum_{i} a_{ij} y_{i} \vert^{s} \Big)^{1/s} \\
\leq & \sup_{y \in B_{\ell_{u'}}}  \Big( \sum_{j} \big\vert \sum_{i} a_{ij} y_{i} \vert^{2} \Big)^{1/2} n^{1/s - 1/2} \, ,
\end{align*}
where $\frac{1}{s} = 1 - (\frac{1}{p_{1}} + \cdots + \frac{1}{p_{m}} )$ and noting that $s<2$. Since $\sum_{j=1}^{n} a_{kj} \overline{a}_{lj}=n \delta_{kl}$ we have, for each $y \in B_{\ell_{u'}}$,
\begin{multline*}
\Big( \sum_{j=1}^{n} \big\vert \sum_{i} a_{ij} y_{i} \vert^{2} \Big)^{1/2} = \Big( \sum_{j=1}^{n}\sum_{i_{1}, i_{2}=1}^{n} a_{i_{1}j} \overline{a}_{i_{2}j} y_{i_{1}} \overline{y}_{i_{2}} \Big)^{1/2}
= \Big(\sum_{i_{1}, i_{2}=1}^{n} \sum_{j=1}^{n} a_{i_{1}j} \overline{a}_{i_{2}j} y_{i_{1}} \overline{y}_{i_{2}} \Big)^{1/2} \\
= n^{1/2}  \Big(\sum_{i=1}^{n} \vert y_{i} \vert^{2} \Big)^{1/2}
\leq n^{1/2}  \Big(\sum_{i=1}^{n} \vert y_{i} \vert^{u'} \Big)^{1/u'} n^{1/2 - 1/u'}  \leq n^{1/u} \, .
\end{multline*}
This altogether gives $\Vert T \Vert \leq n^{\frac{1}{2} + \frac{1}{u} - (\frac{1}{p_{1}} + \cdots + \frac{1}{p_{m}} )}$. On the other hand, $T(e_{i_{1}} , \ldots  e_{i_{m}}) =(a_{1i}, \ldots , a_{ni})$ if $i_{1} = \ldots i_{m} = i$ and
the null vector, otherwise, then $\big( \sum \Vert T(e_{i_{1}} , \ldots  e_{i_{m}}) \Vert_{\ell_q}^{t} \big)^{1/t}=n^{1/t+1/q}$ and, if \eqref{ecuacion t} holds we have
\[
 \frac{1}{t} \leq \frac{1}{2} + \frac{1}{u} - \frac{1}{q} - \big(\frac{1}{p_{1}} + \cdots + \frac{1}{p_{m}} \big) \,.
\]
Hence, $t\ge\rho$ in the case \eqref{ele p-i-b} under the assumption $\frac{1}{p_{1}} + \cdots + \frac{1}{p_{m}} < \frac{1}{2}$.
\end{proof}

\medskip

\noindent By a deep result of Kwapie\'n \cite{Kw68} we know that every operator $v: \ell_{1} \to \ell_{q}$ is $(r,1)$--summing with $\frac{1}{r} = 1 - \big\vert \frac{1}{q} - \frac{1}{2} \big\vert$, and
this $r$ is optimal. For $q=2$ this is Grothendieck's theorem. A straightforward application of Theorem~\ref{main main} with this gives the following.
\begin{proposition}\label{Kwapien}
   Let $1 \leq p_{1} , \ldots , p_{m} \leq \infty$ and $1\leq q \leq \infty$. Then there is $C>0$ such that, for every continuous $m$-linear
$T: \ell_{p_{1}} \times \cdots \times  \ell_{p_{m}} \to \ell_{1}$ with coefficients $(a_{i_{1} , \ldots , i_{m} })$ and every operator $v: \ell_{1} \to \ell_{q}$ we have
\[
  \Big( \sum_{i_{1} , \ldots , i_{m}=1}^{\infty} \Vert v a_{i_{1} , \ldots , i_{m} } \Vert^{\rho} \Big)^{1/\rho} \leq C \Vert T \Vert \, ,
\]
where $\rho$ is given by
\begin{enumerate}
  \item If $1 \leq q \leq 2$ and
      \begin{enumerate}
        \item if $0 \leq \frac{1}{p_{1}} + \cdots + \frac{1}{p_{m}} < 1 - \frac{1}{q}$, then $\rho = \frac{2m}{m+2 - 2( 1/q - (1/p_{1} + \cdots + 1/p_{m}))}$.
	\item  if $1 - \frac{1}{q} \leq \frac{1}{p_{1}} + \cdots + \frac{1}{p_{m}} < \frac{3}{2} - \frac{1}{q}$, then
						  $\rho = \frac{2}{3-2( 1/q+ (1/p_{1} + \cdots + 1/p_{m}))}$.
      \end{enumerate}
  \item If $2 \leq q$ and
     \begin{enumerate}
        \item if $0 \leq \frac{1}{p_{1}} + \cdots + \frac{1}{p_{m}} < \frac{1}{2}$, then $\rho = \frac{m}{1/2+m/q - (1/p_{1} + \cdots + 1/p_{m})}$.
	\item  if $ \frac{1}{2} \leq \frac{1}{p_{1}} + \cdots + \frac{1}{p_{m}} < \frac{1}{2} + \frac{1}{q}$, then
						  $\rho = \frac{1}{1/2+ 1/q - (1/p_{1} + \cdots + 1/p_{m})}$.
 \end{enumerate}
 \end{enumerate}
\end{proposition}

\section{Final comments}
An $m$-linear mapping between Banach spaces $T: X_{1} \times \cdots \times X_{m} \to Y$  is multiple $(t;r_{1}, \ldots , r_{m})$--summing
(see e.g. \cite{Ma03,BoPGVi04}) if there is $K>0$ such that for
every $\big( x_{i_{j}}^{(j)} \big)_{i_{j}=1}^{N_{j}} \subseteq X_{j}$, for $j=1,  \ldots , m$ we have
\[
  \Big( \sum_{i_{1} , \ldots , i_{m}}^{N_{1} , \ldots , N_{m}} \Vert T (x_{i_{1}}^{(1)} , \ldots , x_{i_{m}}^{(m)} ) \Vert^{t}  \Big)^{1/t}
\leq K \prod_{j=1}^{m} \sup_{x^{*}_{j} \in B_{X^{*}_{j}}} \Big( \sum_{i_{j} =1 }^{N_{j} } \vert x^{*}_{j} (x_{i_{j}}^{(j)} ) \vert^{r_{j}} \Big)^{1/r_{j}} \, .
\]
We denote by $\mathcal{L}_{\mathrm{ms}(t;r_{1}, \ldots , r_{m}) }(^{m} X_{1} , \ldots X_{m} ; Y )$ the space of multiple $(t;r_{1}, \ldots , r_{m})$--summing $m$-linear mapppings.
Proceeding as in \cite[Corollary~3.20]{PGVi04} one gets that the following two statements are equivalent
\begin{itemize}
  \item There is a constant $C>0$ such that for every $T\in \mathcal{L}(^{m} \ell_{p_{1}} , \ldots \ell_{p_{m}} ;Y ) $ the following holds
\[
  \Big( \sum_{i_{1} , \ldots , i_{m}} \Vert T(e_{i_{1}} , \ldots , e_{i_{m}} ) \Vert^{t} \Big)^{1/t} \leq C \Vert T \Vert
\]
 \item For all Banach spaces $X_{1} , \ldots , X_{m}$ we have
$$
\mathcal{L} (^{m} X_{1} , \ldots X_{m} ; Y )
=\mathcal{L}_{\mathrm{ms}(t;p_{1}', \ldots , p_{m}') }(^{m} X_{1} , \ldots X_{m} ; Y ).
$$
\end{itemize}
Then all our results have a straightforward interpretation as coincidence results for multiple summing multilinear mappings.\\[3ex]

We have recently learned that some particular cases of some of our results (more precisely Proposition~\ref{main tool} for $q=2$ and the case \eqref{ele p-i-a} in Proposition~\ref{ele p})
have been independently obtained in \cite{AlBaPeSe}.

\bigskip

\subsection*{Acknowledgements} This work was intitated in June 2012, when the first author was visiting Universidad de Valencia and Universidad Polit\'{e}cnica de Valencia
and was finished in June 2013, during a stay of the second author at Universidad de Buenos Aires. Both authors wish to thank all the people in Valencia and Buenos Aires, in and outside
all three Universities who made those visits such a delightful time.

\end{document}